\documentclass{qam-l}

\copyrightinfo{2019}{UT-Battelle}

\usepackage{local-math}
\usepackage{cite}
\usepackage{graphicx}
\usepackage{algorithm}
\usepackage{algorithmic}
\usepackage{xfrac}
\usepackage{placeins}
\usepackage{amsthm}
\usepackage{amsmath}

\makeatletter
\newcommand{\newrestatetheorem}[2]{%
\newtheorem*{restate@#1}{\restate@title}%
\newenvironment{restate#1}[2][]%
{\def\restate@title{#2 \ref{##2}}%
\begin{restate@#1}[##1]}%
{\end{restate@#1}}}
\makeatother

\def\publname{}
\def\currentvolume{}

\newtheorem{theorem}{Theorem}[section]
\newtheorem{corollary}[theorem]{Corollary}
\newrestatetheorem{theorem}{Theorem}
\newrestatetheorem{corollary}{Corollary}
\newtheorem{lemma}[theorem]{Lemma}

\theoremstyle{definition}

\numberwithin{equation}{section}

\newcommand{\avg}{\overline}
\newcommand{\ava}{\avg{\vect a}}
\newcommand{\avb}{\avg{\vect b}}
\newcommand{\avx}{\avg{\vect x}}

\newcommand{\meana}{\avg{a}}
\newcommand{\avy}{\avg{\vect y}}

\newcommand{\vone}{\vect 1^T}

\newcommand{\avaonen}{\ava \, \vone_n}

\newcommand{\LabAddress}{%
  Oak Ridge National Laboratory, P.O.\ Box 2008, MS6054, Oak Ridge, TN, 37831}

\begin{document}

\title{Efficiently updating a covariance matrix and its LDL~decomposition}
\author{Don March}
\address{\LabAddress}
\email{marchdd@ornl.gov}
\thanks{%
  The authors wish to thank Dr.~Wayne Barrett for his helpful comments after
  reading an early draft of this paper.}

\author{Vandy Tombs}
\address{\LabAddress}
\email{tombsvj@ornl.gov}



\maketitle

\begin{abstract}
  Equations are presented which efficiently update or downdate the covariance
  matrix of a large number of $m$-dimensional observations. Updates and
  downdates to the covariance matrix, as well as mixed updates/downdates, are
  shown to be rank-$k$ modifications, where $k$ is the number of new
  observations added plus the number of old observations removed. As a result,
  the update and downdate equations decrease the required number of
  multiplications for a modification to $\Theta((k+1)m^2)$ instead of
  $\Theta((n+k+1)m^2)$ or $\Theta((n-k+1)m^2)$, where $n$ is the number of
  initial observations. Having the rank-$k$ formulas for the updates also allows
  a number of other known identities to be applied, providing a way of applying
  updates and downdates directly to the inverse and decompositions of the
  covariance matrix. To illustrate, we provide an efficient algorithm for
  applying the rank-$k$ update to the LDL decomposition of a covariance matrix.
\end{abstract}

\section{Introduction}
\label{sec:introduction}

Methods for analyzing multidimensional signals frequently involve the
calculation of the covariance matrix of the observation vectors in the sample
set. Being the single parameter which describes the spread of data in a
multivariate Gaussian distribution, the covariance matrix is as fundamental as
the standard deviation is for univariate data, and it frequently plays an
analogous role in analysis.

For example, each pixel in a hyperspectral image records a spectrum that is
sampled across many different wavelengths. Many of the formulas used to analyze
hyperspectral imagery model the background of a scene (i.e.,~the non-target
pixels) as a multivariate Gaussian distribution; the covariance matrix is used
to determine the significance of a deviation from the mean background spectrum.
Equations used in hyperspectral imagery analysis that involve the covariance
matrix include methods for anomaly detection, supervised and unsupervised
classification, and sub-pixel target detection
\cite{manolakis_hyperspectral_2003}.

It is sometimes useful to recalculate the covariance matrix after including
additional observation vectors (an \emph{update}) or after removing a subset of
the original observation vectors (a \emph{downdate}). For example,
\cite{kucuk_anomaly_2015} presents variations of the Reed-Xiaoli~(RX) anomaly
detector which use a local covariance matrix; rather than modeling the
background of a scene using all pixels, the background is calculated using the
pixels within a sliding window. As this window slides one step, many of the
pixels inside the previous window boundary will be inside the new window
boundary, but some additional pixels are included and some of the former pixels
are dropped. The result is a \emph{mixed update/downdate} to the statistics
calculated over the sliding window. In \cite{caefer_improved_2008}, similar
methods that require updating the covariance matrix are used to improve the
performance of target detectors.

While recalculating the covariance matrix can improve algorithm accuracy, a
higher computational cost is incurred as well. Suppose that $X$ is an $m\times
n$ matrix containing $n$ observation vectors, each with $m$ features; let
$\avg{\vect x}$ denote the mean column vector. The sample covariance matrix%
\footnote{An unqualified \emph{covariance matrix} is used throughout to refer to
  the sample covariance matrix.} %
can be calculated as %
\begin{equation*}
  \label{eq:covariance-matrix}
S = %
  \frac{1}{n-1}%
  \paren{X-\avg{\vect{x}}\, \vect 1_n^T}%
  \paren{X-\avg{\vect{x}}\, \vect 1_n^T}^T%
  \hspace{-4pt},\hspace{4pt}
\end{equation*}
which expresses removing the sample mean from the data matrix and then
multiplying the resulting matrix by its transpose. As a result, the covariance
matrix is symmetric positive semidefinite, and positive definite if and only if
$X$ has rank $m$\ (see \cite{golub_matrix_computations_1996}). The number of
operations%
\footnote{When providing the operation counts, we mean addition, subtraction,
  multiplication, and division. When reporting the number of additions and
  multiplications separately, subtractions are counted as addition and divisions
  are counted as multiplication.} %
necessary to calculate the matrix product is $nm^2 - (m^2+m)/2$. Thus, if we add
$k$ observations to the sample set, or remove $k$ observations, then the matrix
product requires $(n\pm k)m^2 - (m^2+m)/2$ operations.

In this paper, we derive update and downdate equations that allow efficient
updates to the covariance matrix and its matrix decompositions. An analogous
procedure is the familiar rule that allows the average of a set to be quickly
updated as observations are added to, or removed from, the sample. We state the
rule both for motivation and since it is used throughout the paper.

\begin{lemma}
\label{lemma:mean-update}
Let $X$ and $Y$ be multisets of real numbers and let $\avg x$ and $\avg y$
denote the arithmetic means of those sets. Then the updated mean is given by
\begin{equation}
\avg{\paren{X \uplus Y}} = %
  \frac{\abs{X}\avg x + \abs Y \avg y}{\abs X + \abs Y}.
\end{equation}
Additionally, if $Y \subset X$, the downdated mean is
\begin{equation}
  \avg{\paren{X\backslash Y}} = %
  \frac{\abs{X}\avg x - \abs Y \avg y}{\abs X - \abs Y}.
\end{equation}
\end{lemma} %
\noindent While a naive calculation of the updated or downdated mean would
require ${(n \pm k - 1)m}$ additions and $m$~multiplications, the original mean
can be reused to avoid much of the work, arriving at the same result after
$km$~additions and $2m$~multiplications. Similarly, the update and downdate
equations stated below allow recalculating the covariance matrix using on the
order of $(k+1)m^2$ operations instead of $(n\pm k)m^2$.

Another familiar procedure is the online update of the covariance of two
variables. A single-pass algorithm for updating the covariance of a dataset was
presented by Bennett, et al.\ in \cite{bennett}: given a dataset $X_1$ of
ordered pairs $x = (u,v)$ with mean $(\avg u, \avg v)$, the covariance of the
updated dataset $X_2 = X_1 \cup \{(s,t)\}$ is given by
\begin{equation}
\label{eq:bennett-update}
\operatorname{Cov}(X_2) = \operatorname{Cov}(X_1) + \tfrac{n-1}{n}(s-\avg u)(t - \avg v).
\end{equation}

In this paper we develop updates and downdates to the covariance matrix that
take the form of rank-$k$ modifications; that is, given a $m\times m$ covariance
matrix $S_1$, recalculating the covariance matrix after adding or removing $k$
vectors to the dataset can be expressed as $S_2 = \alpha S_1 + \beta KK^T$ where
$K$ is a $m \times k$ matrix and $\alpha$ and $\beta$ are scalars. A covariance
matrix could also be updated by applying covariance and standard deviation
updates to the individual entries of the matrix, and such an update does perform
substantially fewer operations than the recalculation of the covariance matrix
using the updated dataset. The benefit of expressing the operation as a rank-$k$
modification is twofold: first, calculating $KK^T$ is an efficient, stable, and
easily parallelizable operation; and second, rank-$k$ modifications can
lead to similarly efficient updates to the inverse or matrix decompositions of
the original matrix.

The following theorem and corollary are the main results of the paper, along
with Theorem~\ref{thm:rank-k-mixed}, which combines the operations enabled by
Theorem~\ref{thm:downdate} and Corollary~\ref{cor:update} into a rank-$k$ mixed
update/downdate.

\begin{theorem}[Rank-$k$ covariance matrix downdate]
\label{thm:downdate}
Let $X_1$ be an $m\times n$ matrix with covariance matrix $S_1$ and let $X_2$ be
formed by deleting $k$ columns of $Y$ from $X_1$. Then the covariance matrix of
$X_2$ is given by the rank-$k$ downdate%
\footnote{These equations assume that the sample covariance matrix is
  calculated with Bessel's correction, that is, dividing variances and
  covariances by $n-1$ when there are $n$ observations sampled. If the
  correction is not applied, the coefficients on $S_2$ and $S_1$ can be
  replaced with $(n-k)$ and $n$, respectively. The fraction in the radical
  remains unchanged.} %
\begin{equation*}
\label{eq:downdate}
\paren{n-k-1}S_2 = \paren{n-1}S_1 - KK^T
\end{equation*}
where
\begin{equation*}
K = Y - \paren{\avy \pm \sqrt{\frac{n}{n-k}}\paren{\avy - \avx_1}\hspace{-2pt}} \vone_k.\hspace{2pt}
\end{equation*}
\end{theorem} %

\begin{corollary}[Rank-$k$ covariance matrix update]\label{cor:update}
Let $X_1$ be an $m\times n$ matrix with covariance matrix $S_1$ and let $Y$
be an $m \times k$ matrix.  Let $X_2$ be formed by appending the columns of $Y$
to $X_1$. Then the covariance matrix of $X_2$ is given by the rank-$k$
update
\begin{equation*}
\paren{n+k-1}S_2 = \paren{n-1}S_1 + KK^T
\end{equation*}
where
\begin{equation*}
K = Y - \paren{\avy \pm \sqrt{\frac{n}{n+k}}\paren{\avy - \avx_1}\hspace{-2pt}} \vone_k.
\end{equation*}
\end{corollary}

The rest of the paper is organized as follows: In Section~\ref{section:lemmas},
we state the necessary notation and lemmas for the proof of the later theorems.
Section~\ref{section:update-theorems} contains the derivation and proof of the
equation for a rank-$k$ covariance matrix update or downdate.
Section~\ref{section:mixed} states and proves the rank-$k$ mixed update/downdate
of the covariance matrix. Finally, in Section~\ref{section:LDL}, we show how the
rank-$k$ updates provide an efficient method for updating covariance matrix
factorizations.

\section{Notation and Lemmas}
\label{section:lemmas}
If $X$ is an $m\times n$ matrix, we write the mean column of $X$ as %
$\avg{\vect x}$ and use $\avg{x}_i$ to denote the $i$th entry in the mean
column. Let $\vect{1}_n$ be a column vector with $n$ entries that all equal $1$.
We frequently use this in expressions such as $X-\vect{a} \vone_n$ to denote
subtracting the column vector $\vect{a}$ from every column of $X$.

There are several key observations used frequently in later proofs.

\begin{lemma}
\label{lemma:2}
Let $A_1$ be an $m\times n$ matrix and let $B$ be an $m\times k$ matrix. If $A_2$
is formed by appending the columns of $B$ to $A_1$ then
\begin{equation}
A_2A_2^T = A_1A_1^T + BB^T.
\end{equation}
If the columns of $B$ are a subset of the columns of $A_1$, and $A_2$ is formed by
deleting the columns of $B$ from $A_1$ then
\begin{equation}
A_2A_2^T = A_1A_1^T - BB^T.
\end{equation}
\end{lemma}
\begin{proof}
The downdate equality can be checked for each entry of the matrix product. Using
$a_{ij}$ to denote the $i,j$ entry in $A_1$,
\begin{align*}
A_2A_2^T
= \sparen{\sum_{\substack{1 \le k \le n, \\ \text{if } \vect a_k \notin B}} a_{ik}a_{jk}}
&= \sparen{
  \sum_{1\le k \le n}a_{ik}a_{jk}
  - \sum_{\substack{1 \le k \le n, \\ \text{if } \vect{a}_k  \in B}}a_{ik}a_{jk}
}     \\
&= \sparen{
  \sum_{1\le k \le n}a_{ik}a_{jk}}
  - \sparen{\sum_{\substack{1 \le k \le n, \\ \text{if } \vect{a}_k  \in B}}a_{ik}a_{jk}
}\\
&= A_1A_1^T - BB^T.
\end{align*}
Showing the update version is similar, but it also follows directly from
swapping the roles of $A_2$ and $A_1$.
\end{proof}

\begin{lemma}
\label{lemma:3}
Let $A$ be an $m\times n$ matrix with mean column vector $\avg{\vect{a}}$. Then
for any compatible column vector $\vect{x}$,
\begin{equation}
\label{eq:l31}
A\paren{\vect{x}\vect 1_n^T}^T
= \paren{\avg{\vect{a}}\hspace{2pt}\vect 1_n^T}\paren{\vect{x}\vect 1_n^T}^T
= n\avg{\vect{a}}\vect{x}^T
\end{equation}
and
\begin{equation}
\label{eq:l32}
\paren{\vect{x}\vect 1_n^T}A^T
= \paren{\vect{x}\vect 1_n^T}\paren{\avg{\vect{a}}\hspace{2pt}\vect 1_n^T}^T
= n\vect{x}\avg{\vect{a}}^T
\end{equation}

\end{lemma}
\begin{proof}
The lemma is a result of
\begin{align*}
A\paren{\vect{x}\vect 1_n^T}^T &=
       \Bigg[\sum_{k=1}^na_{ik}x_j\Bigg] \\
       &=\Bigg[\paren{\sum_{k=1}^n a_{ik}}x_j\Bigg]
       =\Bigg[ n\meana_ix_j \Bigg]
         = n\ava\vect{x}^T
\intertext{and, where $\vect{w}$ is any compatible vector,}
\paren{\vect{w}\vect 1_n^T}\paren{\vect{x}\vect 1_n^T}^T
       &= \Bigg[ \sum_{i=1}^n w_ix_j \Bigg]
       = \Bigg[ n w_i x_j \Bigg]
       = n \vect{w}\vect{x}^T.
\end{align*}
Then setting $\vect{w} = \ava$ shows the second equality in (\ref{eq:l31}),
and~(\ref{eq:l32}) follows from transposing both sides of (\ref{eq:l31}).
\end{proof}

\begin{lemma}
\label{lemma:4}
Let $A$ be an $m\times n$ matrix, let $\ava$ be the mean column of $A$, and let
$s$ and $t$ be real numbers. Then
\begin{equation}
\label{eq:l41}
\paren{A + s\paren{\avaonen}}\paren{A + t\paren{\avaonen}}^T
= AA^T + n\paren{st + s + t}\ava\,\ava^T
\end{equation}
and, in particular, if $s = t = -1$ then
\begin{equation}
\label{eq:l42}
\paren{A - {\avaonen}}\paren{A - {\avaonen}}^T
= AA^T - n\ava\,\ava^T.
\end{equation}
\end{lemma}
\begin{proof}
Let $M = \avaonen$. Then by distributing and applying Lemma~\ref{lemma:3},
\begin{align*}
\paren{A + s\paren{\avaonen}}\paren{A + t\paren{\avaonen}}^T
  &= \paren{A + sM}\paren{A + tM}^T \\
  &= \paren{A+sM}\paren{A^T + tM^T} \\
  &= AA^T + sMA^T + tAM^T + stMM^T \\
  &= AA^T + ns\paren{\ava\,\ava^T} + nt\paren{\ava\,\ava^T} + nst\paren{\ava\,\ava^T} \\
  &= AA^T + n\paren{st + s + t}\paren{\ava\,\ava^T}.
\end{align*}
\end{proof}

\section{Update Theorems}
\label{section:update-theorems}

We now restate and prove the theorem given in the introduction.

\begin{restatetheorem}[Rank-$k$ covariance matrix downdate]{thm:downdate}
Let $X_1$ be an $m\times n$ matrix with covariance matrix $S_1$ and let $Y$ be
an $m \times k$ matrix where the columns of $Y$ are a subset of the columns of
$X_1$. Let $X_2$ be formed by deleting the columns of $Y$ from $X_1$. Let the
mean columns of $X_1$, $X_2$, and $Y$ be $\avx_1$, $\avx_2$, and $\avy$,
respectively. Then the covariance matrix of $X_2$ is given by the rank-$k$
downdate %
\begin{equation}
\label{eq:downdate}
\paren{n-k-1}S_2 = \paren{n-1}S_1 - KK^T
\end{equation}
where
\begin{equation}
\label{eq:K-in-downdate}
K = Y - \paren{\avy \pm \sqrt{\frac{n}{n-k}}\paren{\avy - \avx_1}\hspace{-2pt}} \vone_k.
\end{equation}
\end{restatetheorem}

\begin{proof}
The calculations for the original and subsequent covariance matrices are
\begin{align*}
(n-1)S_1   &= \paren{X_1 - \avx_1 \, \vone_{n}} \paren{X_1 - \avx_1 \, \vone_{n}}^T \\
(n-k-1)S_2 &= \paren{X_2 - \avx_2 \, \vone_{n-k}} \paren{X_2 - \avx_2 \, \vone_{n-k}}^T\hspace{-4pt}.\hspace{4pt}
\end{align*}
The goal is to reuse as much of the calculation of $S_1$ as possible in the
calculation in $S_2$. The mean column is going to be subtracted from each column
of $X_2$ to calculate the covariance matrix. Thus, it will make no difference if
we first shift $X_2$ by subtracting the same vector from each column (before
calculating and subtracting the mean); in particular, we can subtract $\avx_1$
from each column:
\begin{align*}
(n&-k-1)S_2 \\
&= \paren{X_2 - \avx_2 \, \vone_{n-k}} \paren{X_2 - \avx_2 \, \vone_{n-k}}^T \\
&= \paren{\paren{X_2 - \avx_1 \, \vone_{n-k}} - \paren{\avx_2 -  \avx_1}\vone_{n-k}}
   \paren{\paren{X_2 - \avx_1 \, \vone_{n-k}} - \paren{\avx_2 - \avx_1}\vone_{n-k}}^T \\
&= \paren{A - \ava\, \vone_{n-k}} \paren{A - \ava\, \vone_{n-k}}^T\hspace{-4pt},\hspace{4pt}
\end{align*}
where $\ava = \paren{\avx_2 - \avx_1}$ is the mean column of %
$A = X_2 - \avx_1 \, \vone_{n-k}$. Applying the special case~(\ref{eq:l42}) of
Lemma~\ref{lemma:4} gives
\begin{align}
(n-k-1)S_2
  &= AA^T - \paren{n-k}\ava\, \ava^T,
\label{eq:Q2-initial}
\end{align}
where the $(n-k)$ factor is due to the number of columns in $A$ (which is the
same size as $X_2$). Using Lemma~\ref{lemma:mean-update}, we can write $\ava$ as
\begin{equation}
\ava
= \avx_2 - \avx_1
= \frac{n\avx_1 - k\avy}{n-k}  - \avx_1
= \paren{\frac{-k}{n-k}}\paren{\avy - \avx_1}.
\label{eq:vecta}
\end{equation}
Note that the columns of $A$ are a subset of the columns of $X_1 - \avx_1 \,
\vone_{n}$; the columns that have been removed are $Y - \avx_1 \, \vone_k$. Thus, we can use Lemma~\ref{lemma:2} to rewrite $AA^T$ as
\begin{align}
AA^T
  &= \paren{X_2 - \avx_1\, \vone_{n-k}}\paren{X_2 - \avx_1\, \vone_{n-k}}^T
\nonumber \\
  &= (X_1 - \avx_1 \, \vone_{n})(X_1 - \avx_1 \, \vone_{n})^T
    - (Y - \avx_1 \, \vone_{k})(Y - \avx_1 \, \vone_{k})^T \nonumber \\
  &= (n-1)S_1 - (Y - \avx_1 \, \vone_{k})(Y - \avx_1 \, \vone_{k})^T.
\label{eq:AA}
\end{align}

Combining (\ref{eq:vecta}) and (\ref{eq:AA}) with the right-hand side of
(\ref{eq:Q2-initial}),
\begin{align}
(n&-k-1)S_2 \nonumber \\
  &= (n-1)S_1 - (Y - \avx_1 \, \vone_{k})(Y - \avx_1 \, \vone_{k})^T
  - \tfrac{k^2}{n-k}\paren{\avy - \avx_1}\paren{\avy-\avx_1}^T
  \nonumber \\
  &= (n-1)S_1 -
    \paren{BB^T + \tfrac{k^2}{n-k}\,\avb\, \avb^T},
\label{eq:Q2}
\end{align}
where $\avb = \avy - \avx_1$ is the mean column of $B = Y - \avx_1\vone_k$. The
last term in (\ref{eq:Q2}) is in the form of the right hand side of
(\ref{eq:l41}) in Lemma~\ref{lemma:4} with %
\mbox{$k(st + s + t) = \frac{k^2}{n-k}$} %
since $B$ has $k$ columns. We wish to use the lemma to factor (\ref{eq:Q2}) as %
$(n-1)S_1 - KK^T$ %
where %
$K = B + c\paren{\avb\, \vone_k}$, %
so we set $s = t = c$. Solving for $c$ in %
$k\paren{c^2 + 2c} = \frac{k^2}{n-k}$ %
gives %
$c = -1 \pm \sqrt{\frac{n}{n-k}}$, and we now have
\begin{align*}
(n&-k-1)  S_2 = (n-1)S_1 - KK^T
\end{align*}
where
\begin{align*}
K  = {B + c\paren{\avb \vone_k}}
   &= {Y - \avx_1\vone_k + c\paren{\avy - \avx_1} \vone_k} \\
   &= {Y - \paren{\avy \pm \sqrt{\frac{n}{n-k}}\paren{\avy - \avx_1}\hspace{-2pt}} \vone_k}.
\end{align*}
\end{proof}

Note that the equation $k(st + s + t) = \frac{k^2}{n-k}$ has many different
solutions. For example, we can also choose $s=0$ and $t = \frac{k}{n-k}$ and
then apply Lemma~\ref{lemma:4} to arrive at
\begin{align}
  \label{eq:mixed-update}
\paren{n-k-1}S_2
  &= \paren{n-1}S_1 - \paren{Y-\avx_1\vone_k}
    \paren{Y - \avx_1\vone_k + \tfrac{k}{n-k}\paren{\avy - \avx_1}\vone_k}^T \nonumber \\
  &= \paren{n-1}S_1 - \paren{Y-\avx_1\vone_k}
    \paren{Y - \tfrac{1}{n-k}\paren{n\avx_1 -k\avx_1 - k\avy + k\avx_1}\vone_k}^T \nonumber \\
  &= \paren{n-1}S_1 - \paren{Y-\avx_1\vone_k}
    \paren{Y - \avx_2\vone_k}^T
\end{align}
which is~(\ref{eq:bennett-update}) generalized to the
entire covariance matrix.

\begin{restatecorollary}[Rank-$k$ covariance matrix update]{cor:update}
Let $X_1$ be an $m\times n$ matrix with covariance matrix $S_1$ and let $Y$
be an $m \times k$ matrix.  Let $X_2$ be formed by appending the columns of $Y$
to $X_1$. Then the covariance matrix of $X_2$ is given by the rank-$k$
update
\begin{equation*}
\paren{n+k-1}S_2 = \paren{n-1}S_1 + KK^T
\end{equation*}
where %
\begin{align}
K &= Y - \paren{\avy \pm \sqrt{\frac{n+k}{n}}\paren{\avy - \avx_2}\hspace{-2pt}} \vone_k
  \label{eq:update-K1} \\
  &= Y - \paren{\avy \pm \sqrt{\frac{n}{n+k}}\paren{\avy - \avx_1}\hspace{-2pt}} \vone_k.
  \label{eq:update-K2}
\end{align}
\end{restatecorollary}

\begin{proof}
Since the $k$ columns of $Y$ are a subset of the $n+k$ columns of $X_2$, the
update equation with $K$ as stated in~(\ref{eq:update-K1}) follows from
Theorem~\ref{thm:downdate} by swapping the roles of $X_1$ and $X_2$. The
alternate equation for $K$ in~(\ref{eq:update-K2}) follows by using Lemma
\ref{lemma:mean-update} to substitute for $\avx_2$ and simplifying.
\end{proof}

Using the equations in this section, updating and downdating the covariance
matrix both require ${(k+1)m^2} + {(3k+4)m} + 4$ operations,%
\footnote{The number of multiplications is %
  \mbox{$\paren{(k+2)m^2 + (k+6)m + 2}\hspace{-1pt}/2$} %
  and the number of additions is %
  \mbox{$\paren{km^2 + (5k+2)m + 6}\hspace{-1pt}/2$}.\vspace{3pt}} %
plus a single square root. A~naive calculation of the new covariance matrix, on the
other hand, requires %
${(n \pm k)m^2} + {(2n + (1\pm 2)k + 2)m} + 3$
operations.%
\footnote{The number of multiplications is %
  \mbox{$\paren{(n \pm k+1)m^2 + (n \pm k+5)m}\hspace{-1pt}/2$} %
  and the number of additions is %
  \mbox{$\paren{(n \pm k-1)m^2 + (3n +(2 \pm 3)k-1)m + 2}\hspace{-1pt}/2$}.\vspace{3pt}} %

\section{Mixed updates and downdates}
\label{section:mixed}

When updating a statistic to include new observations, it is common to want to
remove other observations at the same time. Using the update and downdate
equations from Section~\ref{section:update-theorems}, it is possible to perform
a mixed update/downdate by simply performing an update and downdate in either
order while skipping the intermediate scaling of the covariance matrix. For
example, using the form of the covariance matrix update expressed in
(\ref{eq:mixed-update}), an update followed by a downdate is given by
\begin{alignat*}{3}
\paren{n+k_u-1}S_u &=\, &(n-1) &S_1 &+ &\paren{Y_u-\avx_1\vone_{k_u}} \paren{Y_u - \avx_u\vone_{k_u}}^T \\
\paren{n+k_u-k_d-1}S_2 &=\, &(n+k_u-1) &S_u &- &\paren{Y_d-\avx_u\vone_{k_d}} \paren{Y_d - \avx_2\vone_{k_d}}^T
\end{alignat*}
where the matrix $Y_u$ holds the $k_u$ update observations with mean $\avy_u$;
the downdate equivalents are $Y_d$, $k_d$, and $\avy_d$; and %
$\avx_u$ is the average of the updated data matrix ${X_u = [ X_1 \,\,\, Y_u]}$.

The combined update/downdate
\begin{align}
  \label{eq:obvious-mixed-update}
  \paren{n+k_u-k_d-1}S_2 = \paren{n-1}S_1 \nonumber 
  &+ \paren{Y_u-\avx_1\vone_{k_u}} \paren{Y_u - \avx_u\vone_{k_u}}^T \\
  &- \paren{Y_d-\avx_u\vone_{k_d}} \paren{Y_d - \avx_2\vone_{k_d}}^T
\end{align}
is clearly more efficient. In fact, it is also possible to calculate $S_2$ without
referring to the intermediate data mean $\avx_u$.

\begin{theorem}[Mixed update/downdate]\label{thm:mixed}
  Let $X_1$ be an $m\times n_1$ matrix with covariance matrix $S_1$ and let
  $Y_u$ and $Y_d$ be $m \times k_u$ and $m \times k_d$ matrices where the
  columns of $Y_d$ are a subset of the columns of $X_1$. Let $X_2$ be the
  $m\times n_2$ data matrix formed by deleting the columns of $Y_d$ from $X_1$
  and appending the columns of $Y_u$. Let the mean columns of $X_1$ and $X_2$ be
  $\avx_1$ and $\avx_2$, respectively. Then the covariance matrix of $X_2$ is
  given by
\begin{align*}
\label{eq:mixed}
  \paren{n_2-1}S_2 = \paren{n_1-1}S_1
  &+ \paren{Y_u-\avx_1\vone_{k_u}} \paren{Y_u - \avx_2\vone_{k_u}}^T \\
  &- \paren{Y_d-\avx_1\vone_{k_d}} \paren{Y_d - \avx_2\vone_{k_d}}^T\hspace{-4pt}.\hspace{4pt}
\end{align*}
\end{theorem}

\begin{proof}
  We begin by noticing that the combined update/downdate in equation
  (\ref{eq:obvious-mixed-update}) is nearly in the form we would like;
  specifically, the first term of $\paren{Y_u-\avx_1\vone_{k_u}} \paren{Y_u -
    \avx_u\vone_{k_u}}^T$ is in the desired form. To get $\avx_2$ in
  the second term, we do the following:
\begin{align*}
\paren{Y_u-\avx_1\vone_{k_u}} \paren{Y_u -\avx_u\vone_{k_u}}^T 
&= \paren{Y_u-\avx_1\vone_{k_u}} \paren{Y_u - \paren{\avx_2 + \paren{\avx_u - \avx_2}}\vone_{k_u}}^T \\
&= \paren{Y_u-\avx_1\vone_{k_u}} \paren{\paren{Y_u - \avx_2\vone_{k_u}} - \paren{\avx_u - \avx_2}\vone_{k_u}}^T \\
&=\paren{Y_u-\avx_1\vone_{k_u}} \paren{Y_u - \avx_2\vone_{k_u}}^T - \paren{Y_u-\avx_1\vone_{k_u}}\paren{\avx_u - \avx_2}^T\hspace{-4pt}.\hspace{4pt}
\intertext{We can then apply Lemma~\ref{lemma:3} to get}
\paren{Y_u-\avx_1\vone_{k_u}} \paren{Y_u -\avx_u\vone_{k_u}}^T
&=\paren{Y_u-\avx_1\vone_{k_u}} \paren{Y_u - \avx_2\vone_{k_u}} -
k_u\paren{\avy_u - \avx_1}\paren{\avx_u - \avx_2}^T\hspace{-4pt}.\hspace{4pt}
\end{align*}
Similarly, the downdate portion of the combined update/downdate can be written
as
\begin{align*}
\paren{Y_d-\avx_u\vone_{k_d}} \paren{Y_d - \avx_2\vone_{k_d}}^T
&= \paren{Y_d-\paren{\avx_1 + \paren{\avx_u - \avx_1}}\vone_{k_d}} \paren{Y_d - \avx_2\vone_{k_d}}^T \\
&= \paren{Y_d-\avx_1\vone_{k_d}} \paren{Y_d - \avx_2\vone_{k_d}}^T -
k_d\paren{\avx_u - \avx_1}\paren{\avy_d - \avx_2}^T\hspace{-4pt}.
\end{align*}
Thus, all that remains is to show
\begin{equation}
\label{eq:mean-equality}
    k_u\paren{\avy_u - \avx_1}\paren{\avx_u - \avx_2}^T = k_d\paren{\avx_u -
\avx_1}\paren{\avy_d - \avx_2}^T\hspace{-4pt}.
\end{equation}
By Lemma~\ref{lemma:mean-update}, we have 
\begin{align*}
\avx_u = \frac{n_1\avx_1 + k_u\avy_u}{n_1+k_u} \qquad
\text{and} \qquad
\avx_2 =
\frac{(n_1+k_u)\avx_u - k_d\avy_d}{n_1+k_u-k_d}
\end{align*}
which give the following:
\begin{align*}
k_u(\avy_u - \avx_1) &= (n_1+k_u)(\avx_u - \avx_1) \\
k_d(\avy_d - \avx_2) &= (n_1+k_u)(\avx_u - \avx_2).
\end{align*}
Substituting these into (\ref{eq:mean-equality}) shows the equality to be true.
\end{proof}

The mixed update/downdate theorem is a more direct and aesthetically pleasing
way of calculating the modified covariance matrix
equation~(\ref{eq:obvious-mixed-update}), and we find it more surprising
than~(\ref{eq:mixed-update}). However, it requires subtracting two different
means ($\avx_1$ and $\avx_2$) from $Y_u$ as well as $Y_d$. A more efficient
mixed update/downdate would be of the form
\begin{equation}
(n + k_u - k_d - 1)S_2 = (n-1)S_1 + (Y_u - \vect{z}_u)(Y_u - \vect{z}_u)^T - (Y_d - \vect{z}_d)(Y_d - \vect{z}_d)^T.
\end{equation}
An additional slight improvement would be if $\vect{z}_u = c\vect{z}_d$ for some
scalar $c$, and the best that we can hope for is $\vect{z}_u = \vect{z}_d$.

It turns out that such a factorization is possible.
Theorem~\ref{thm:rank-k-mixed} subsumes the rank-$k$ update and downdate
equations from Section~\ref{section:update-theorems}.

\begin{theorem}[Rank-$k$ mixed update/downdate]
  \label{thm:rank-k-mixed}
  Let $X_1$ be an $m\times n_1$ matrix with covariance matrix $S_1$ and let
  $Y_u$ and $Y_d$ be $m \times k_u$ and $m \times k_d$ matrices where the
  columns of $Y_d$ are a subset of the columns of $X_1$. Let $X_2$ be the
  $m\times n_2$ data matrix formed by deleting the columns of $Y_d$ from $X_1$
  and appending the columns of $Y_u$. Let the mean columns of $X_1$ and $X_2$ be
  $\avx_1$ and $\avx_2$, respectively. Then the covariance matrix of $X_2$ is
  given by the rank-$k$ ($k = k_u + k_d$) mixed update/downdate %
\begin{align}
\label{eq:rank-k-mixed}
  \paren{n_2-1}S_2 = \paren{n_1-1}S_1 + K_uK_u^T - K_dK_d^T
\end{align}
where
$K_u = \paren{Y_u - \vect{z}\vone_{k_u}}$, $K_d = \paren{Y_d - \vect{z}\vone_{k_d}}$,
$z = \avx_1 - c(\avx_2 - \avx_1)$, and
\begin{equation*}
c = \begin{cases}
   \dfrac{n_2 \pm \sqrt{n_1n_2}}{n_2 - n_1} &\text{if $k_u \neq k_d$}
\vspace{.75em}
 \\
   \hfill \frac{1}{2} \hfill & \text{if $k_u = k_d$}.
\end{cases}
\end{equation*}
\end{theorem}

\begin{proof}
  From the definition of $\avx_2$,
  \begin{equation*}
    \avx_2 - \avx_1 
    = \frac{{k_u\paren{\avy_u-\avx_1} - k_d\paren{\avy_d - \avx_1}}}{n_2}
  \end{equation*}
and
  \begin{equation*}
    0 = \avx_2 - \avx_2 = n_1\paren{\avx_1 - \avx_2} + k_u\paren{\avy_u - \avx_2} - k_d\paren{\avy_d - \avx_2},
  \end{equation*}
which give these two equalities:
  \begin{align}
    n_2\paren{\avx_2 - \avx_1} &= {k_u\paren{\avy_u-\avx_1} - k_d\paren{\avy_d - \avx_1}} \label{eq:replacement1}
\\
    n_1\paren{\avx_2 - \avx_1} &= {k_u\paren{\avy_u-\avx_2} - k_d\paren{\avy_d - \avx_2}}. \label{eq:replacement2}
  \end{align}
Applying Lemma~\ref{lemma:3} with a compatible vector $\vect{a}$, these
become:
\begin{align*}
  n_2\paren{\paren{\avx_2 - \avx_1}\vone_{n_2}}(\vect{a}\vone_{n_2})^T &=
\paren{Y_u - \avx_1\vone_{k_u}}\paren{\vect{a}\vone_{k_u}}^T
- {\paren{Y_d - \avx_1\vone_{k_d}}\paren{\vect{a}\vone_{k_d}}^T} \\
  n_2(\vect{a}\vone_{n_2})\paren{\paren{\avx_2 - \avx_1}\vone_{n_2}}^T &=
\paren{\vect{a}\vone_{k_u}}\paren{Y_u - \avx_1\vone_{k_u}}^T
- \paren{\vect{a}\vone_{k_d}}{\paren{Y_d - \avx_1\vone_{k_d}}^T}\hspace{-4pt}.
\end{align*}
Motivated by these factorizations, along with the form of the downdate seen in
(\ref{eq:K-in-downdate}), we define %
$K_u= \paren{Y_u - \paren{\avx_1 + c\paren{\avx_2 - \avx_1}}\vone_{k_u}}$ %
and examine $K_uK_u^T$:
\begin{align}
  K_uK_u^T
  &= \paren{Y_u - \paren{\avx_1 + c\paren{\avx_2 - \avx_1}}\vone_{k_u}}
    \paren{Y_u - \paren{\avx_1 + c\paren{\avx_2 - \avx_1}}\vone_{k_u}}^T \nonumber \\
  &= \paren{Y_u - \paren{\avx_1 + c\paren{\avx_2 - \avx_1}}\vone_{k_u}} \paren{Y_u- \paren{\avx_2 + \paren{c-1}\paren{\avx_2 - \avx_1}}\vone_{k_u}}^T \nonumber \\
  &= \paren{\paren{Y_u - \avx_1\vone_{k_u}} - c\paren{\avx_2 - \avx_1}\vone_{k_u}} \paren{\paren{Y_u- \avx_2\vone_{k_u}} - \paren{c-1}\paren{\avx_2 - \avx_1}\vone_{k_u}}^T\hspace{-4pt}. \nonumber \\
\intertext{Then applying Lemma~\ref{lemma:3},}
  \begin{split}
  K_uK_u^T &=\, \paren{Y_u - \avx_1\vone_{k_u}}\paren{Y_u - \avx_2\vone_{k_u}}^T + \paren{c^2-c}k_u\paren{\avx_2 - \avx_1}\paren{\avx_2-\avx_1}^T \\
    &\qquad -ck_u\paren{\avx_2 - \avx_1}\paren{\avy_u - \avx_2}^T - \paren{c-1}k_u\paren{\avy_u - \avx_1}\paren{\avx_2 - \avx_1}^T\hspace{-4pt}.
    \label{eq:mixed-Ku}
  \end{split}
\intertext{Likewise, if
$K_d= \paren{Y_d - \paren{\avx_1 + c\paren{\avx_2 - \avx_1}}\vone_{k_d}}$ %
then}
    \begin{split}
    K_dK_d^T
      &= \paren{Y_d - \avx_1\vone_{k_d}}\paren{Y_d - \avx_2\vone_{k_d}}^T + \paren{c^2-c}k_d\paren{\avx_2 - \avx_1}\paren{\avx_2-\avx_1}^T \\
      & \qquad -ck_d\paren{\avx_2 - \avx_1}\paren{\avy_d - \avx_2}^T - \paren{c-1}k_d\paren{\avy_d - \avx_1}\paren{\avx_2 - \avx_1}^T\hspace{-4pt}.
    \label{eq:mixed-Kd}
    \end{split}
\end{align}
Subtracting the first term on the right side of~(\ref{eq:mixed-Kd}) from the
first term of~(\ref{eq:mixed-Ku}) is equal to the mixed update/downdate of
Theorem~\ref{thm:mixed}. Therefore, $K_uK_u^T - K_dK_d^T$ is the mixed
update/downdate as long as all of the remaining terms of~(\ref{eq:mixed-Ku})
cancel those of~(\ref{eq:mixed-Kd}). The difference between these terms is
\begin{align}
   &\paren{c^2-c}\paren{k_u-k_d}\paren{\avx_2 - \avx_1}\paren{\avx_2-\avx_1}^T
   - c\paren{\avx_2 - \avx_1}\paren{k_u\paren{\avy_u - \avx_2} - k_d\paren{\avy_d - \avx_2}}^T \nonumber \\
  &\qquad- \paren{c-1}\paren{k_u\paren{\avy_u - \avx_1} - k_d\paren{\avy_d - \avx_1}}\paren{\avx_2 - \avx_1}^T\hspace{-4pt}, \nonumber \\
\intertext{ which, after substituting (\ref{eq:replacement1}) and
(\ref{eq:replacement2}), becomes}
  &\paren{c^2-c}\paren{n_2-n_1}\paren{\avx_2 - \avx_1}\paren{\avx_2-\avx_1}^T
    - n_1c\paren{\avx_2 - \avx_1}\paren{\avx_2 - \avx_1}^T \nonumber \\
  &\qquad - n_2\paren{c-1}\paren{\avx_2 - \avx_1}\paren{\avx_2 - \avx_1}^T \nonumber \\
  = &\paren{\paren{c^2-c}\paren{n_2-n_1}- n_1c - (c-1)n_2}
      \paren{\avx_2 - \avx_1}\paren{\avx_2 - \avx_1}^T \nonumber\\
  = &\paren{\paren{n_2 - n_1}c^2 - \paren{2n_2}c + n_2}
      \paren{\avx_2 - \avx_1}\paren{\avx_2 - \avx_1}^T\hspace{-4pt}.
      \label{eq:mixed-rank-k-difference}
\end{align}
Finally, we solve for values of $c$ that would make
~(\ref{eq:mixed-rank-k-difference}) zero, giving $c=\frac{1}{2}$ if $k_u = k_d$
and
in all other cases
\begin{equation*}
c = \frac{2n_2 \pm \sqrt{\paren{2n_2}^2 -
4n_2\paren{n_2-n_1}}}{2\paren{n_2-n_1}} = \frac{n_2 \pm \sqrt{n_1n_2}}{n_2-n_1}.
\end{equation*}
\end{proof}

Briefly, we note that an alternate way of writing~(\ref{eq:rank-k-mixed}). Let
$K = [K_u \,\,\, iK_d]$, where $i$ is the imaginary unit, in which case
$\paren{n_2 - 1}S_2 = \paren{n_1-1}S_1 + KK^T$\hspace{-2pt}.

\section{Updating the LDL Factorization of the Covariance Matrix}
\label{section:LDL}

Given the equations for rank-$k$ updates and downdates in
Theorems~\ref{thm:downdate} and~\ref{thm:rank-k-mixed}, it is possible to apply
many existing update theorems and numerical libraries to the covariance matrix.
For example, the Sherman-Morrison-Woodbury formula enables the application of a
rank-$k$ update to a matrix as a rank-$k$ update to the inverse of that
matrix~\cite{hager_updating_1989}. Similarly, a rank-$k$ modification to a
positive semidefinite matrix can be applied directly to the LDL and Cholesky
decompositions, bypassing the calculation of the updated non-factored
matrix~\cite{davis}.

In this section, we will show how the rank-$k$ covariance matrix modifications
can be used to update or downdate the LDL~decomposition of a covariance matrix.
Frequently, the reason to calculate the covariance matrix is to multiply by its
inverse. For numerical stability and optimization, this should typically be
implemented using an alternative computation, such as calculating the LDL~
decomposition and then using forward substitution.

Let $S_1 = L_1D_ 1L_1^T$ be a covariance matrix and its LDL decomposition. As
stated in Section~\ref{sec:introduction}, the covariance matrix is positive
semidefinite, and therefore has both Cholesky and LDL decompositions.%
\footnote{A Cholesky and LDL decomposition exists for any covariance matrix,
  but the decomposition is not guaranteed to be stable, unique, or to exclude
  zero entries along the diagonal unless the covariance matrix is positive
  definite \cite{golub_matrix_computations_1996}.} %
The goal is to find an efficient bulk downdate for the decomposition by taking
advantage of the rank-$k$ downdate to $S_1$ provided by
Theorem~\ref{thm:downdate}. As is the case with Corollary~\ref{cor:update}, the
downdate to the LDL decomposition immediately generalizes to the update as well.

Applying the downdate theorem,
\begin{align*}
    \paren{n-k-1}S_2 = \paren{n-1}L_1D_1L_1^T - KK^T
\end{align*}
where $K$ is as defined in the theorem. Following \cite{gill}, we reuse $L_1$
and $D_1$ to calculate the LDL decomposition of $S_2$. We can write $K = L_1P$
where $P$ is found using forward substitution, arriving at
\begin{align*}
    \paren{n-k-1}S_2 = \paren{n-1}L_1DL_1^T - (L_1P)(L_1P)^T.
\end{align*}
Letting $D' = \paren{n-1}D$ and simplifying,
\begin{align*}
    \paren{n-k-1}S_2 = L_1(D' - PP^T)L_1^T.
\end{align*}
We can find the LDL decomposition of $D'-PP^T$, namely
$D' - PP^T = \widetilde{L}\widetilde{D}\widetilde{L}^T$ (see
\cite{gill} for the proof of existence).
Setting $L_2 = L_1\widetilde{L}$ and $D_2 = \widetilde{D}$ gives $S_2 =
L_2D_2L_2$, the LDL decomposition of the downdated covariance matrix.

Gill~et~al.\ introduce a number of algorithms in \cite{gill} for updating the
LDL factorization that can be applied to Theorem~\ref{thm:downdate} and
Corollary~\ref{cor:update}. A one-pass algorithm of Method~C1 in~\cite{gill} is
provided by Algorithm~1 in~\cite{davis}. We introduce a modification of this
algorithm that provides a rank-$k$ update or downdate to the LDL factorization
of the covariance matrix. The modifications incorporate the calculation of $K$
from~(\ref{eq:K-in-downdate}) into the algorithm so that only one column of $K$
is required to be in memory at any time.

Let $X_1$, $Y$, $X_2$, and $S_1$ be as defined in Theorem \ref{thm:downdate} (or
Corollary \ref{cor:update}). Let the mean columns of $X_1$ and $Y$ be denoted as
$\avx_1$ and $\avy$, respectively. Let $d_{ij}, y_{ij},$ and $\ell_{ij}$ be the
$i,j$th entries of $D, Y$, and $L$, respectively. Given the factorization $S_1 =
LDL^T$, the matrices $L$ and $D$ are overwritten with the new factors of $S_2$
by Algorithm~\ref{alg:downdate}.

\begin{algorithm}[hb]
\caption{Covariance matrix LDL factorization update or downdate}
\label{alg:downdate}
\begin{algorithmic}
    \STATE $\phi = 1$ for update or $-1$ for downdate
    \STATE $D = (n-1)D$
    \STATE $c = \sqrt{\frac{n}{n + \phi k}}$
    \STATE $\boldsymbol{z } = \avy - c(\avy -\avx_1)$
    \FOR{$j=1$ to $k$}
        \STATE $\alpha = \phi$
        \FOR{$i = 1$ to $m$}
            \STATE $y_{ij} = y_{ij} + z_i$
        \ENDFOR
        \FOR{$i = 1$ to $m$}
            \STATE $\widetilde d = d_{ii}$
            \STATE $\gamma = y_{ij}/(\alpha \hspace{0.08em} d_{ii} + y_{ij}^2)$
            \STATE $d_{ii} = d_{ii} + y_{ij}^2/\alpha$
            \STATE $\alpha = \alpha + y_{ij}^2/\widetilde d$
            \FOR{$p = i+1$ to $m$}
                \STATE $y_{pj} = y_{pj} - y_{ij}\hspace{0.08em}\ell_{pi}$
                \STATE $\ell_{pi} = \ell_{pi} + \gamma\hspace{0.08em} y_{pj}$
            \ENDFOR
        \ENDFOR
    \ENDFOR
    \STATE $D = \tfrac{1}{n + \phi k - 1}D$
\end{algorithmic}
\end{algorithm}

Algorithm~\ref{alg:downdate} performs $2km^2 + (8k+5)m + 4$ total operations,%
\footnote{The number of multiplications is %
  \mbox{$km^2 + (4k+4)m + 1$} %
  and the number of additions is %
  \mbox{$km^2 + (4k+1)m +1$.}} %
plus a single square root. The naive calculation of the LDL decomposition of the
new covariance matrix requires $(n \pm k)^3/3$ operations (Algorithm 4.1.2 in
\cite{golub_matrix_computations_1996}), in addition to the operations required
to update the covariance matrix itself. This algorithm can also be easily
extended to implement the rank-$k$ mixed update/downdate from
Theorem~\ref{thm:rank-k-mixed}; the outer loop runs twice, first iterating over
the columns of $Y_u$ with $\phi=1$, and then over the columns of $Y_d$ with
$\phi=-1$.

\section*{Acknowledgements}
This manuscript has been authored by UT-Battelle,~LLC, under contract
DE-AC05-00OR22725 with the US Department of Energy~(DOE). The US~government
retains and the publisher, by accepting the article for publication,
acknowledges that the US~government retains a nonexclusive, paid-up,
irrevocable, worldwide license to publish or reproduce the published form of
this manuscript, or allow others to do so, for US~government purposes. DOE~will
provide public access to these results of federally sponsored research in
accordance with the DOE~Public Access Plan
(http://energy.gov/downloads/doe-public-access-plan).

This research was supported in part by an appointment to the Oak Ridge National
Laboratory Post-Master's Research Associate Program, sponsored by the
US~Department of Energy and administered by the Oak~Ridge Institute for Science
and Education.

\bibliography{2020-march_tombs-covariance.bib}{}
\bibliographystyle{amsplain}

\end{document}